\documentclass[12pt]{amsart}
\usepackage{graphicx}
\usepackage{amssymb}
\newtheorem{thm}{Theorem}[section]
\newtheorem{cor}[thm]{Corollary}

\newtheorem{prop}[thm]{Proposition}
\theoremstyle{definition}
\newtheorem{defn}[thm]{Definition}
\theoremstyle{remark}
\newtheorem{rem}[thm]{Remark}
\numberwithin{equation}{section}
\newtheorem{exa}[thm]{Example}

\newcommand{\h}{\mathcal{H}}

\newcommand{\K}{\mathcal{K}}

\DeclareMathOperator{\Ran}{ran}
\DeclareMathOperator{\Span}{span}
\DeclareMathOperator{\Ker}{ker}

\begin{document}
\title[Representation properties of $G$-frames]{Properties of bounded representations for $G$-frames}%
\author {Fatemeh Ghobadzadeh, Yavar Khedmati and Javad Sedghi Moghaddam}%
\address{
\newline
\indent Department of Mathematics
\newline
\indent Faculty of Sciences
\newline \indent   University of Mohaghegh Ardabili
\newline \indent  Ardabil 56199-11367
\newline \indent Iran}
\email{gobadzadehf@yahoo.com}
\email{khedmati.y@uma.ac.ir, khedmatiy.y@gmail.com}
\email{j.smoghaddam@uma.ac.ir}
\begin{abstract}
Due to the importance of frame representation by a bounded operator in dynamical sampling, researchers studied the frames of the form $\{T^{i-1} f\}_{i\in \mathbb{N}}$, which $f$ belongs to separable Hilbert space $\h$ and $T\in B(\h)$, and investigated the properties of $T$. Given that $g$-frames include the wide range of frames such as fusion frames, the main purpose of this paper is to study the characteristics of the operator $T$ for $g$-frames of the form $\{\Lambda T^{i-1} \in B(\h,\K):i\in \mathbb{N}\}$.
\end{abstract}
\maketitle
\section{{\textbf Introduction}}
Duffin and Schaeffer introduced an extension of orthonormal bases for separable Hilbert space $\h$ named frames \cite{DS}, which in spite of producing $\h$, is not necessarily linearly independent. Frames are important tools in the signal/image processing \cite{Anto, Balador, Desti}, data compression \cite{Data2, Data1}, dynamical sampling \cite{ald1, ald2} and etc. 
\begin{defn}
A sequence $F=\{f_i\}_{i\in \mathbb{N}}$ in $\h$ is called a  frame for $\h$, if there exist two constants $A_F, B_F> 0$ such that
\begin{equation*}\label{abc}
A_F \|f\|^2\leq\sum_{i\in \mathbb{N}}|\langle f,f_i\rangle |^2 \leq B_F \|f\|^2,\quad f\in \h.
\end{equation*}
\end{defn}
For more on frames we refer to \cite{c4, HDL}.
 \par Aldroubi et al. introduced the concept of dynamical sampling to examine sequences of the form $\{T^{i-1}f\}_{i\in\mathbb{N}}\subset\h$, that spans $\h$ for $T\in B(\h)$. As frames span the space, researchers have studied the frames $F=\{f_i\}_{i\in\mathbb{N}}$ for infinte dimensional Hilbert space $\h$ that can be represented by $T$, i.e. $F=\{T^{i-1}f_1\}_{i\in\mathbb{N}}$ \cite{ald1,ald2}. Christensen et al. have shown that the only frames with bounded representations are those which are linearly independent and the kernel of their synthesis operators is invariant under right-shift operator $\mathcal{T}:\ell^2(\h,\mathbb{N})\rightarrow\ell^2(\h,\mathbb{N})$ defined by 
 \begin{align*}
 \mathcal{T}(\{c_i\}_{i\in\mathbb{N}})=(0,c_1,c_2,...), 
 \end{align*}
 where $\ell^2(\h,\mathbb{N})=\big\{\{g_i\}_{i\in \mathbb{N}}:g_i\in\h,\sum_{i\in \mathbb{N}}\|g_i\|^2<\infty\big\}$,
 such as orthonormal bases and Riesz bases \cite{main}. They have also explored the relationship between frame representation and its duals. For the applications of frames, they established that frame representations were preserved under some perturbations. Results \cite[Theorem 7]{ald2} and 
 \cite[Proposition 3.5]{chretal} are shown that the sequence $\{T^{i-1}f_1\}_{i\in\mathbb{N}}$ is not a frame, whenever $T$ is unitary or compact. Also, Lemma 2.1 and Proposition 2.3 of \cite{RashNej} indicate $\Ran{T}$ is close and give some equivalent conditions for $T$ to be surjective.
\par In 2006, Sun introduced a generalization of frames, named $g$-frames \cite{ws} which are including some extensions and types of frames such as frames of subspaces \cite{caku}, fusion frames \cite{fus,fusion}, oblique frames \cite{obli}, a class of time-frequency localization operators and generalized translation invariant (GTI) \cite{dtuykhmj}. Therefore, some concepts presented in frames such as duality, stability and Riesz-basis were also studied in $g$-frames \cite{wsper}.
\par Throughout this paper, $J$ is countable set, $\mathbb{N}$ is natural numbers and $\mathbb{C}$ is complex numbers, $\h$ and $\K$ are separable Hilbert spaces, $Id_{\h}$ denotes the identity operator on $\h$,  $B(\h)$ and $GL(\h)$ denote the set of bounded linear operators and invertible bounded linear operators on $\h$, respectively. Also, we will apply $B(\h,\K)$ for the set of bounded linear operators from $\h$ to $\K$. We use $\ker T$ and $\Ran T$ for the null space and range $T\in B(\h)$, respectively. 
Now, we summarize some facts about $g$-frames from \cite{Anaj and Rah, ws}. For more on related subjects to $g$-frames, we refer to \cite{gu, mus, AnajFR}. 
 \begin{defn}
We say that $\Lambda=\{\Lambda_{i}\in B(\h,\K_{i}):i\in \mathbb{N}\}$ is a generalized frame for $\h$  with respect to $\{\K_{i}:i\in \mathbb{N}\}$, or simply $g$-frame, if there are two constants $0<A_\Lambda\leq B_\Lambda<\infty$ such that
\begin{equation}\label{cgframe}
A_\Lambda\|f\|^{2}\leq\sum_{i\in \mathbb{N}}\|\Lambda_i f\|^{2}\leq B_\Lambda\|f\|^{2},\; f\in \h.
\end{equation}
We call $A_\Lambda,B_\Lambda$ the lower and upper $g$-frame bounds, respectively.
$\Lambda$ is called a tight $g$-frame if $A_\Lambda=B_\Lambda,$ and a
Parseval $g$-frame if $A_\Lambda=B_\Lambda=1.$ 
If for each $i\in \mathbb{N},$ $\K_{i}=\K,$ then, $\Lambda$ is called a $g$-frame for $\h$ with respect to $\K$. 
Note that for a family $\{\K_i\}_{i\in \mathbb{N}}$ of Hilbert spaces, there exists a Hilbert space $\K=\oplus_{i\in \mathbb{N}}\K_i$ such that for all $i\in \mathbb{N}$, $\K_i\subseteq\K$, where $\oplus_{i\in \mathbb{N}}\K_i$ is the direct sum of $\{\K_i\}_{i\in \mathbb{N}}$.
A family $\Lambda$ is called $g$-Bessel if the right hand inequality in (\ref{cgframe}) holds for all $f\in \h,$ in this case, $B_\Lambda$ is called the $g$-Bessel bound.
\end{defn}
 \begin{exa}\cite{ws}
 Let $\{f_i\}_{i\in \mathbb{N}}$ be a frame for $\h$. Suppose that 
 $\Lambda=\{\Lambda_{i}\in B(\h,\mathbb{C}):i\in \mathbb{N}\}$, where  $$\Lambda_i f=\langle f,f_i\rangle,\quad f\in\h.$$ 
 It is easy to see that $\Lambda$ is a $g$-frame.
 \end{exa}
For a $g$-frame $\Lambda$, there exists a unique positive and invertible operator $S_\Lambda:\h\rightarrow \h$ such that
\begin{eqnarray*}\label{di}
S_\Lambda f=\sum_{i\in \mathbb{N}}\Lambda_i^*\Lambda_i f,\quad f\in\h,
\end{eqnarray*}
 and $A_\Lambda. Id_{\h}\leq S_\Lambda\leq B_\Lambda. Id_{\h}.$ Consider the space
$$\Big(\sum_{i\in \mathbb{N}}\oplus\K_i\Big)_{\ell^2}=\Big\{\{g_i\}_{i\in \mathbb{N}}: g_i\in \K_i,\:i\in \mathbb{N}\: and\: \sum_{i\in \mathbb{N}}\|g_i\|^2<\infty\Big\}.$$
It is clear that, $\Big(\sum_{i\in \mathbb{N}}\oplus\K_i\Big)_{\ell^2}$ is a Hilbert space with pointwise operations and with the inner product given by $$\big\langle \{f_i\}_{i\in \mathbb{N}},\{g_i\}_{i\in \mathbb{N}}\big\rangle=\sum_{i\in \mathbb{N}}\langle f_i,g_i\rangle.$$
For a $g$-Bessl $\Lambda$, the synthesis operator $T_\Lambda:\Big(\sum_{i\in \mathbb{N}}\oplus\K_i\Big)_{\ell^2}\rightarrow\h$ is defined by 
\begin{align*}
T_\Lambda\big(\{g_i\}_{i\in \mathbb{N}}\big)=\sum_{i\in \mathbb{N}}\Lambda_i^*g_i.
\end{align*}
The adjoint of $T_\Lambda$, $T_\Lambda^*:\h\rightarrow\Big(\sum_{i\in \mathbb{N}}\oplus\K_i\Big)_{\ell^2}$ is called the analysis operator of $\Lambda$ and is as follow
\begin{align*}
T_\Lambda^* f=\{\Lambda_i f\}_{i\in \mathbb{N}},\quad f\in\h.
\end{align*}
It is obvious that $S_\Lambda=T_\Lambda T_\Lambda^*.$
For a $g$-frame $\Lambda=\{\Lambda_{i}\in B(\h,\K_i):i\in \mathbb{N}\}$, the sequence
$\widetilde{\Lambda}=\{\widetilde{\Lambda}:=\Lambda_i S_\Lambda^{-1}\in B(\h,\K_i):i\in \mathbb{N}\}$ is a $g$-frame  with lower and upper $g$-frame bounds $\frac{1}{B_{\Lambda}}$ and $\frac{1}{A_{\Lambda}}$, respectively, which is called canonical dual of $\Lambda$. For $g$-Bessel sequences $\Lambda$ and $\Theta$, we consider
$S_{\Lambda\Theta}:=T_\Lambda T_\Theta^*$.
 \begin{defn} Consider a sequence $\Lambda=\{\Lambda_{i}\in B(\h,\K_i):i\in \mathbb{N}\}$.
 \begin{enumerate}
 \item[(i)] 
 We say that $\Lambda$ is $g$-complete if $\{f: \Lambda_if=0, i\in \mathbb{N}\}=\{0\}$.
 \item[(ii)]
 We say that $\Lambda$ is a $g$-Riesz sequence if there are two constants $0<A_\Lambda\leq B_\Lambda<\infty$ such that for any finite set $\{g_i\}_{i\in I_n}$, 
\begin{align*}
A_\Lambda\sum_{i\in I_n}\|g_i\|^{2}\leq\|\sum_{i\in I_n}\Lambda_i^*g_i\|^{2}\leq B_\Lambda\sum_{i\in I_n}\|g_i\|^{2},\quad g_i\in \K_i.
\end{align*}
 \item[(iii)]
 We say that $\Lambda$ is a $g$-Riesz basis if $\Lambda$ is $g$-complete and $g$-Riesz sequence.
\item[(iv)] We say that $\Lambda$ is a $g$-orthonormal basis if it satisfies the following:
\begin{align*}
&\langle\Lambda_i^*g_i,\Lambda_j^* g_j\rangle=\delta_{i,j}\langle g_i,g_j\rangle,\quad i,j\in \mathbb{N}, g_i\in\K_i, g_j\in \K_j,
\\&\sum_{i\in \mathbb{N}}\|\Lambda_i f\|^2=\|f\|^2,\quad f\in\h.
\end{align*}
\end{enumerate}
\end{defn}
 A $g$-Riesz basis $\Lambda=\{\Lambda_{i}\in B(\h,\K_i):i\in \mathbb{N}\}$ is $g$-biorthonormal with respect to its canonical dual $\widetilde{\Lambda}=\{\widetilde{\Lambda}:=\Lambda_i S_\Lambda^{-1}\in B(\h,\K_i):i\in \mathbb{N}\}$ in the following sense
 \begin{align*}
 \langle\Lambda_i^*g_i,{\widetilde{\Lambda_j}}^* g_j\rangle=\delta_{i,j}\langle g_i,g_j\rangle,\quad i,j\in \mathbb{N}, g_i\in\K_i, g_j\in \K_j.
 \end{align*}
\begin{thm}\cite{AnajFR}\label{u.n}
Let $\Lambda=\{\Lambda_{i}\in B(\h,\K_i):i\in \mathbb{N}\}$ be a $g$-frame and $\Theta=\{\Theta_{i}\in B(\h,\K_i):i\in \mathbb{N}\}$  be a  $g$-orthonormal basis. Then there is a onto bounded operator $V :\h \rightarrow \h$ such that
$\Lambda_i=\Theta_i V^*$, for all $i\in\mathbb{N}$. If $\Lambda$ is a $g$-Riesz basis, then $V$ is invertible. If $\Lambda$ is a $g$-orthonormal bases, then $V$ is unitary.
\end{thm}
\begin{thm}\cite{ws}\label{framegframe}
 Let for $i\in \mathbb{N}$, $\{e_{i,j}\}_{j\in J_i}$ be an orthonormal basis for $\K_i$. Sequence $\Lambda=\{\Lambda_{i}\in B(\h,\K_i):i\in \mathbb{N}\}$ is a $g$-frame (respectively, $g$-Bessel family, $g$-Riesz basis, $g$-orthonormal basis) if and only if 
$\{\Lambda_i^*e_{i,j}\}_{i\in \mathbb{N}, j\in J_i}$ is a frame (respectively, Bessel sequence, Riesz basis, orthonormal basis).
\end{thm}
Now we summarize some results of article \cite{gframerep} in which we generalize the results of articles \cite{main, chretal} to introduce the representation of $g$-frames with bounded operators.
\begin{rem}\label{frem}
Consider a frame
$F=\{f_i\}_{i\in{\mathbb{N}}}=\{T^{i-1}f_1\}_{i\in{\mathbb{N}}}$ for
$\h$ with $T\in{B(\h)}$. For the $g$-frame $\Lambda=\{\Lambda_{i}\in
{B(\h,\mathbb{C})}: i\in{\mathbb{N}}\}$ where 
\begin{align*}
\Lambda_i f=\langle f,f_i\rangle,\quad f\in\h,
\end{align*}
we have
\begin{align*}
\Lambda_{i+1}f=\langle f,f_{i+1}\rangle=\langle
f,Tf_{i}\rangle=\langle T^*f,f_{i}\rangle=\Lambda_{i}T^*f,\quad
f\in{\h}.
\end{align*}
Therefore, $\Lambda_i=\Lambda_1
(T^*)^{i-1}, i\in{\mathbb{N}}$. Conversely, if we consider a $g$-frame
$\Lambda=\{\Lambda_i\in{B(\h,\mathbb{C})}:i\in{\mathbb{N}}\}=\{\Lambda_{1}T^{i-1}:i\in{\mathbb{N}}\}$ for
$T\in{B(\h)}$, then by the Riesz representation theorem, $\Lambda_{i}f=\langle f,f_{i}\rangle, i\in{\mathbb{N}}$ and
$f,f_i\in{\h}$, where $F=\{f_{i}\}_{i\in{\mathbb{N}}}$ is a frame that $f_i=(T^*)^{i-1}f_1, i\in\mathbb{N}$.
\end{rem}
Now, we have been motivated to study $g$-frames $\Lambda=\{\Lambda_i\in B(\h,\K):i\in\mathbb{N}\},$ where
$\Lambda_i=\Lambda_1 T^{i-1}$ with $T\in{B(\h)}$.
\begin{defn}
We say that a $g$-frame $\Lambda=\{\Lambda_i\in B(\h,\K):i\in\mathbb{N}\}$ has a representation if there is a
$T\in{B(\h)}$ such that $\Lambda_{i}=\Lambda_1 T^{i-1},
i\in{\mathbb{N}}$. In the affirmative case, we say that $\Lambda$ is
represented by $T$.
\end{defn}
The following theorem shows that for $g$-frames $\Lambda=\{\Lambda_1T^{i-1}:i\in\mathbb{N}\}$, the boundedness of $T$ is equivalent to the invariance of $\ker T_\Lambda$ under the right-shift operator.
\begin{thm}\label{MT}
Let $\Lambda=\{\Lambda_{i}\in{B(\h,\K)}: i\in{\mathbb{N}}\}$ be a
$g$-frame such that for every finite set 
$\{g_{i}\}_{i\in{I_{n}}}\subset\K$,
$\sum_{i\in{I_{n}}}\Lambda_{i}^{*}g_{i}=0$ implies $g_{i}=0$ for
every $i\in{I_{n}}$. 
Suppose that $\ker T_{\Lambda}$ is
invariant under the right-shift operator. Then,
$\Lambda$ is represented by  $T\in{B(\h)}$, where $\|T\|\leq\sqrt{B_{\Lambda}A_{\Lambda}^{-1}}$.
\end{thm}
\begin{cor}\label{oth}
Every $g$-orthonormal and $g$-Riesz bases has a representation.
\end{cor}
\begin{rem}\label{gl}
Consider a $g$-frame $\Lambda=\{\Lambda_{i}\in B(\h,\K):i\in \mathbb{N}\}$ which is represented by $T$. For $S\in GL(\h)$, the family 
$\Lambda S=\{\Lambda_{i} S\in B(\h,\K):i\in \mathbb{N}\}$ is a $g$-frame \cite[Corollary 2.26]{AnajFR}, which is represented by $S^{-1}TS$.
\end{rem}  
In this paper, we generalize some recent results of  \cite{chretal,RashNej} to investigate properties of representations for $g$-frames with bounded operators.
\section{{\textbf G-frame representation properties}}
In this section, we examine some properties of operator representations of $g$-frames, including being closed range, injective, unitary and compact.
\par In the following results, we first specify the range of adjoint of $g$-frame operator representations to indicates that the range of operator representations is closed. Then, we get necessary and sufficient conditions for $g$-frames to have injective operator representations.

\begin{thm}\label{cloran}
Let $\Lambda=\{\Lambda_i\in B(\h,\K):i\in \mathbb{N}\}$ be a $g$-frame that is represented by $T$. Then $\Ran T^*=\overline{\Span}\{T^*\Lambda_i^*e_j\}_{i\in{\mathbb{N}},j\in{J}}$, where
$\{e_j\}_{j\in J}$ is an orthonormal basis for $K$, and $\Ran T$ is close.
\end{thm}
\begin{proof}
By Theorem \ref{framegframe}, $\{\Lambda_i^*e_j\}_{i\in\mathbb{N},j\in J}$ is a frame for $\h$, and so for every $f\in\h$, we have
\begin{equation*}
T^*f=T^*\Big(\sum_{i\in\mathbb{N},j\in J}c_{ij}\Lambda_i^*e_j\Big)=\sum_{i\in\mathbb{N},j\in J}c_{ij}T^*\Lambda_i^*e_j.
\end{equation*}
Thus, $\Ran T^*\subseteq\overline{\Span}\{T^*\Lambda_i^*e_j\}_{i\in\mathbb{N},j\in J}:=\h_0$. On the other hand, since $\{T^*\Lambda_i^*e_j\}_{i\in{\mathbb{N}},j\in{J}}$, is a frame for $\h_0$, we have
\begin{equation*}
g=\sum_{i\in\mathbb{N},j\in J}d_{ij}T^*\Lambda_i^*e_j=T^*\Big(\sum_{i\in\mathbb{N},j\in J}c_{ij}\Lambda_i^*e_j\Big),\quad g\in\h_0.
\end{equation*}
Then $\Ran T^*=\h_0$ is close and so $\Ran T$ is close.
\end{proof}
\begin{prop}\label{inject}
Let $\Lambda=\{\Lambda_1T^{i-1}\in B(\h,\K):i\in \mathbb{N}\}$ be a $g$-frame such that $\|\Lambda_1\|<\sqrt{A_\Lambda}$. Then $T$ is injective.
\end{prop}
\begin{proof}
For every $f\in\h$,
\begin{equation*}
A_\Lambda\|f\|^2\leq\sum_{i\in\mathbb{N}}\|\Lambda_1T^{i-1}f\|^2\leq\|\Lambda_1\|^2\big(\|f\|^2+\sum_{i\in\mathbb{N}}\|T^if\|^2\big),
\end{equation*}
thus $\sum_{i\in\mathbb{N}}\|T^if\|^2\geq\big(\frac{A_\Lambda}{\|\Lambda_1\|^2}-1\big)\|f\|^2$ and since
$\frac{A_\Lambda}{\|\Lambda_1\|^2}-1>0$, $T$ is injective.
\end{proof}
\begin{thm}
Let $\Lambda=\{\Lambda_i\in B(\h,\K):i\in \mathbb{N}\}$ be a $g$-frame that is represented by $T$. Then the following are equivalent.
\begin{enumerate}
\item[(i)] $T$ is injective.
\item[(ii)] $\Ran(S_\Lambda^{-1}\Lambda_1^*)\cap\Ker T=\{0\}$.
\item[(iii)] $\Ran\Lambda_1^*\subseteq\Ran T^*$.
\end{enumerate}
\end{thm}
\begin{proof}
$(i)\Rightarrow(ii)$ and $(i)\Rightarrow(iii)$ are clear.\\
$(ii)\Rightarrow(i)$ Suppose that $T$ is not injective. Then there exists $0\neq f\in\Ker T$. We get 
\begin{equation*}
f=\sum_{i\in\mathbb{N}}S_\Lambda^{-1}\Lambda_i^*\Lambda_i f
=S_\Lambda^{-1}\Lambda_1^*\Lambda_1f+\sum_{i\in\mathbb{N}}S_\Lambda^{-1}\Lambda_{i+1}^*\Lambda_i Tf=S_\Lambda^{-1}\Lambda_1^*\Lambda_1f.
\end{equation*}
So $f\in\Ran(S_\Lambda^{-1}\Lambda_1^*)$, which is a contradiction.\\
$(iii)\Rightarrow(i)$ For any $f\in\h$, we have
\begin{align*}
f=\sum_{i\in\mathbb{N}}\Lambda_i^*\Lambda_i S_\Lambda^{-1}f
&=\Lambda_1^*\Lambda_1 S_\Lambda^{-1}f+\sum_{i\in\mathbb{N}}T^*\Lambda_i^*\Lambda_{i+1} S_\Lambda^{-1}f\\
&=\Lambda_1^*\Lambda_1 S_\Lambda^{-1}f+T^*\Big(\sum_{i\in\mathbb{N}}\Lambda_i^*\Lambda_{i+1} S_\Lambda^{-1}f\Big).
\end{align*}
Since $\Ran\Lambda_1^*\subseteq\Ran T^*$, $f\in\Ran T^*$. Therefore $T^*$ is surjective, and so $T$ is injective.
\end{proof}
The main purpose of the reminder of the paper is to show that the operator representation of $g$-frames can not be unitary and compact.
\begin{thm}\label{conv}
Let $\Lambda=\{\Lambda_1T^{i-1}\in B(\h,\K):i\in \mathbb{N}\}$ be a $g$-frame. Then for every $f\in\h$, $T^n f\rightarrow 0$ as $n\rightarrow\infty$.
\end{thm}
\begin{proof}
For every $n\in\mathbb{N}$ and $f\in\h$, we have
\begin{equation}\label{first}
A_\Lambda\|T^nf\|^2\leq\sum_{i\in\mathbb{N}}\|\Lambda_1T^{i-1+n}f\|^2=\sum_{i=n}^\infty\|\Lambda_1T^if\|^2.
\end{equation}
On the other hand, by $\sum_{i\in\mathbb{N}}\|\Lambda_1T^{i-1}f\|^2\leq B_\Lambda\|f\|^2$, we get $\sum_{i=n}^\infty\|\Lambda_1T^if\|^2\rightarrow 0$
 as $n\rightarrow\infty$. Therefore, by the inequality \eqref{first}, we conclude that $T^nf\rightarrow 0$ as $n\rightarrow\infty$.
\end{proof}
\begin{cor}\label{unitary}
For every unitary operator $T$ and every $\Lambda_1\in B(\h,\K)$, the sequence $\Lambda=\{\Lambda_1T^{i-1}\in B(\h,\K):i\in \mathbb{N}\}$ can not be a $g$-frame. 
\end{cor}
\begin{proof}
For every $f\in\h$,
\begin{equation}\label{second}
\|f\|=\|(T^*)^nT^nf\|\leq\|T^*\|^n\|T^nf\|=\|T^nf\|.
\end{equation}
If $\Lambda$ is a $g$-frame, then by Theorem \ref{conv}, $T^nf\rightarrow 0$ as $n\rightarrow\infty$, and so by the inequality \eqref{second}, $\|f\|\rightarrow 0$, that is a contradiction.
\end{proof}
\begin{cor}
Let $\Lambda=\{\Lambda_i\in B(\h,\K):i\in \mathbb{N}\}$ and $\Theta=\{\Theta_i\in B(\h,\K):i\in \mathbb{N}\}$
be two $g$-orthonormal bases. Then for every $\Gamma_1\in B(\h,\K)$, the sequence 
$\Gamma=\{\Gamma_1 S_{\Lambda\Theta}^{i-1}\in B(\h,\K):i\in\mathbb{N}\}$ is not a $g$-frame.
\end{cor}
\begin{proof}
By Theorem \ref{u.n}, there exists a unitary operator $U\in B(\h)$ such that $\Theta_i=\Lambda_iU$. We have
\begin{equation*}
S_{\Lambda\Theta}S_{\Lambda\Theta}^*=T_\Lambda T_\Theta^* T_\Theta T_\Lambda^*
=T_\Lambda T_\Lambda^* UU^*T_\Lambda T_\Lambda^*=S_\Lambda Id_\h S_\Lambda=Id_\h,
\end{equation*}
and similary $S_{\Lambda\Theta}^*S_{\Lambda\Theta}=Id_\h$. So $S_{\Lambda\Theta}$ is a unitary operator
on $\h$ and by Corolary \ref{unitary}, $\Gamma$ is not a $g$-frame for every $\Gamma_1\in B(\h,\K)$.
\end{proof}
\begin{prop}
Let $\h_1$ and $\h_2$ be two Hilbert spaces. Assume that $T\in B(\h_1)$, $S\in B(\h_2)$ and $\Lambda\in B(\h_1,\K)$, $\Theta\in B(\h_2,\K)$ such that $T=V^{-1}SV$ and $\Theta V=\Lambda$ for some $V\in GL(\h_1,\h_2)$. Then $\{\Lambda T^{i-1}\in B(\h_1,\K):i\in \mathbb{N}\}$ is a $g$-frame, if and only if $\{\Theta S^{i-1}\in B(\h_2,\K):i\in \mathbb{N}\}$ is a $g$-frame. In the affirmative case $V$ is unique.
\end{prop}
\begin{proof}
For every $f\in\h_1$, we have
\begin{align*}
\sum_{i\in\mathbb{N}}\|\Lambda T^{i-1}f\|^2=\sum_{i\in\mathbb{N}}\|\Theta V(V^{-1}SV)^{i-1}f\|^2&=\sum_{i\in\mathbb{N}}\|\Theta VV^{-1}S^{i-1}Vf\|^2\\
&=\sum_{i\in\mathbb{N}}\|\Theta S^{i-1}Vf\|^2.
\end{align*}
Since $V\in GL(\h_1,\h_2)$, the sequence $\{\Lambda T^{i-1}\in B(\h_1,\K):i\in \mathbb{N}\}$ is a $g$-frame, if and only if $\Lambda=\{\Theta S^{i-1}\in B(\h_2,\K):i\in \mathbb{N}\}$ is a $g$-frame. Also, by Theorem \ref{framegframe}, there exists $\{c_{ij}\}_{i\in\mathbb{N},j\in J}\in \ell^2(\mathbb{C},\mathbb{N})$ such that
\begin{align*}
(V^*)^{-1}f&=(V^*)^{-1}\Big(\sum_{i\in\mathbb{N},j\in J}c_{ij}(T^{i-1})^*\Lambda^*e_j\Big)\\
&=(V^*)^{-1}\Big(\sum_{i\in\mathbb{N},j\in J}c_{ij}V^*(S^{i-1})^*(V^{-1})^*V^*\Theta^* e_j\Big)\\
&=\sum_{i\in\mathbb{N},j\in J}c_{ij}(S^{i-1})^*\Theta^* e_j,
\end{align*}
which $\{e_j\}_{j\in J}$ is an orthonormal basis for $\K$.
\end{proof}
\begin{prop}
Let $\Lambda=\{\Lambda_i\in B(\h,\K):i\in\mathbb{N}\}$ be a $g$-frame. If for $\Theta\in B(\h,\K)$ the sequence 
$\{\Theta S_{\Lambda}^{i-1}\in B(\h,\K):i\in \mathbb{N}\}$ is a $g$-frame, then $A_\Lambda<1$.
\end{prop}
\begin{proof}
The proof is the same as the proof of the \cite[Proposition 2.7]{RashNej}.
\end{proof}

In \cite[Corollary 2.4]{RashNej}, it has been shown that for Riesz basis $\{T^{i-1}f_1\}_{i\in\mathbb{N}}$ the operator $T$ can not be surjective. While the following examples show that for the 
$g$-Riesz basis $\{\Lambda_1T^{i-1}\in B(\h,\K):i\in \mathbb{N}\}$, the operator $T$ can be injective.
\begin{exa}
\begin{enumerate}
\item[(i)] For $\Lambda_1\in GL(\h)$, the set $\{\Lambda_1\}$ is a $g$-Riesz basis which is represented by $Id_\h$.
\item[(ii)] By  \cite[Corollary 2.4]{RashNej}, for a Riesz basis $F=\{T^{i-1}f_1\}$, $T$ is not surjective. Consider $\Lambda=\{\Lambda_i\in B(\h,\mathbb{C}):i\in \mathbb{N}\}$, where
$\Lambda_i f=\langle f,f_i\rangle$. By Remark \ref{frem}, $\Lambda$ is represented by $T^*$ which is not injective. On the other hand, for any $i\in\mathbb{N}$, $\Lambda_i^*(1)=f_i$, and therefore 
by Theorem \ref{framegframe}, $\Lambda$ is a $g$-Riesz basis.
\end{enumerate}
\end{exa}
Theorem \cite[Proposition 3.5]{chretal} and \cite[Proposition 2.2]{RashNej} show that for frame $\{T^{i-1}f_1\}_{i\in\mathbb{N}}$, the operator $T$ can not be compact. In the following, we show that in the finite space K 
for $g$-frame $\{\Lambda_1T^{i-1}\in B(\h,\K):i\in \mathbb{N}\}$ the operator $T$ is not compact as well. By giving example, we show that this is not generally true.
\begin{prop}
Let $\Lambda=\{\Lambda_1T^{i-1}\in B(\h,\K):i\in \mathbb{N}\}$ be a $g$-frame, where $\K$ is a finite-dimensional Hilbert space. Then $T$ is not compact.
\end{prop}
\begin{proof}
Let $\{e_j\}_{j\in J}$ be an orthonormal basis for $\K$ and $T$ be compact. By Theorem \ref{cloran}, 
$\Ran T^*=\overline{\Span}\{\Lambda_{i+1}^*e_j\}_{i\in \mathbb{N},j\in J}$, and therefore by \cite[Lemma 2.5.1]{c4}, there exists $T^\dagger\in B(\h)$ such that $T^*T^\dagger=Id_{\Ran T^*}$. Since $T$ is compact, $T^*$ is compact and so $\Ran T^*$ is finite-dimensional. Consequently, 
$\overline{\Span}\{\Lambda_i^*e_j\}_{i\in \mathbb{N},j\in J}$ is finite-dimensional and so by Theorem
\ref{framegframe}, $\h$ is finite-dimensional, that is a contradiction.
\end{proof}
\begin{exa}
Consider $\Lambda_1=Id_ {\ell^2(\h,\mathbb{N})}$ and $T:\ell^2(\h,\mathbb{N})\rightarrow\ell^2(\h,\mathbb{N})$, that is defined by 
$T\{a_j\}_{j\in J}=(\alpha a_1,0,0,...)$ for a scalar $\alpha$ with $|\alpha|<1$. It is clear that $T$ is compact and 
$\Lambda=\{\Lambda_1T^{i-1}\in B(\ell^2(\h,\mathbb{N})):i\in \mathbb{N}\}$ is a $g$-frame. In fact, for every 
$\{a_j\}_{j\in J}\in\ell^2(\h,\mathbb{N})$, we have
\begin{align*}
\|\{a_j\}_{j\in J}\|^2&\leq\sum_{i\in\mathbb{N}}\|\Lambda_1T^{i-1}\{a_j\}_{j\in J}\|^2
=\sum_{i\in\mathbb{N}}\|T^{i-1}\{a_j\}_{j\in J}\|^2\\
&=\|\{a_j\}_{j\in J}\|^2+\sum_{i\in\mathbb{N}}\| T^i\{a_j\}_{j\in J}\|^2\\
&=\|\{a_j\}_{j\in J}\|^2+\sum_{i\in\mathbb{N}}\|(\alpha^i a_1,0,0,...)\|^2\\
&\leq\frac{1}{1-\alpha^2}\|\{a_j\}_{j\in J}\|^2.
\end{align*}
\end{exa}

\begin{thm}\label{j1}
Let $\Lambda=\{\Lambda_i\in B(\h,\K):i\in \mathbb{N}\}$ be a $g$-Riesz sequence and
 $\Theta=\{\Theta_i\in B(\h,\K):i\in \mathbb{N}\}$ be a sequence of operators, where 
 $\alpha:=\sum_{i\in\mathbb{N}}\|\Lambda_i-\Theta_i\|\|\Lambda_1 S_\Lambda^{-1}\|<1$ and 
 $\beta:=\sum_{i\in\mathbb{N}}\|\Lambda_i-\Theta_i\|^2<\infty$. Then $\Theta$ is a $g$-Riesz sequence.
\end{thm}
\begin{proof}
For every $\{g_i\}_{i\in {\mathbb{N}}}\in\ell^2(\K,\mathbb{N})$, we have
\begin{align*}
\|\sum_{i\in\mathbb{N}}\Theta_i^*g_i\|
&=\|\sum_{i\in\mathbb{N}}(\Theta_i^*-\Lambda_i^*)g_i+\sum_{i\in\mathbb{N}}\Lambda_i^*g_i\|\\
&\leq\sum_{i\in\mathbb{N}}\|\Theta_i^*-\Lambda_i^*\|\|g_i\|+\|\sum_{i\in\mathbb{N}}\Lambda_i^*g_i\|\\
&\leq\Big(\sum_{i\in\mathbb{N}}\|\Theta_i-\Lambda_i\|^2\Big)^{\frac{1}{2}}\|\{g_i\}_{i\in{\mathbb{N}}}\|_{\ell^2(\K,\mathbb{N})}
+\sqrt{B_\Lambda}\|\{g_i\}_{i\in{\mathbb{N}}}\|_{\ell^2(\K,\mathbb{N})}\\
&\leq\big(\sqrt{\beta}+\sqrt{B_\Lambda}\big)\|\{g_i\}_{i\in{\mathbb{N}}}\|_{\ell^2(\K,\mathbb{N})}.
\end{align*}
So for well-defined operator $U:\h\rightarrow\h$, defined by 
\begin{equation*}
Uf=\sum_{i\in\mathbb{N}}\Theta_i^*\Big(\sum_{j\in J}\langle f,S_\Lambda^{-1}\Lambda_i^*e_j\rangle e_j\Big),
\end{equation*}
we have
\begin{align*}
\|Uf\|&=\Big\|\sum_{i\in\mathbb{N}}\Theta_i^*\Big(\sum_{j\in J}\langle f,S_\Lambda^{-1}\Lambda_i^*e_j\rangle e_j\Big)\Big\|\\
&\leq\big(\sqrt{\beta}+\sqrt{B_\Lambda}\big)\Big\|\Big\{\sum_{j\in J}\langle f,S_\Lambda^{-1}\Lambda_i^*e_j\rangle e_j\Big\}_{i\in\mathbb{N}}\Big\|_{\ell^2(\K,\mathbb{N})}\\
&=\big(\sqrt{\beta}+\sqrt{B_\Lambda}\big)\Big\|\Big\{\sum_{j\in J}\langle f,P_MS_\Lambda^{-1}\Lambda_i^*e_j\rangle e_j\Big\}_{i\in\mathbb{N}}\Big\|_{\ell^2(\K,\mathbb{N})}\\
&=\big(\sqrt{\beta}+\sqrt{B_\Lambda}\big)\Big\|\Big\{\sum_{j\in J}\langle P_Mf,S_\Lambda^{-1}\Lambda_i^*e_j\rangle e_j\Big\}_{i\in\mathbb{N}}\Big\|_{\ell^2(\K,\mathbb{N})}\\
&=\big(\sqrt{\beta}+\sqrt{B_\Lambda}\big)\big\|\big\{\Lambda_i S_\Lambda^{-1}P_Mf\big\}_{i\in\mathbb{N}}\big\|_{\ell^2(\K,\mathbb{N})}\\
&\leq\frac{\sqrt{\beta}+\sqrt{B_\Lambda}}{\sqrt{A_\Lambda}}\|P_M f\|
\leq\frac{\sqrt{\beta}+\sqrt{B_\Lambda}}{\sqrt{A_\Lambda}}\|f\|,
\end{align*}
where $M=\overline{\Span}\{\Lambda_i^*e_j\}_{i\in\mathbb{N},j\in J}$. Note that the operator $U$ on $M$ is equal to $S_{\Theta\Lambda}S_\Lambda^{-1}$. On the other hand, for every $k\in\mathbb{N}$, we have
\begin{align*}
\langle U\Lambda^*_kg,f\rangle&=\sum_{i\in\mathbb{N}} \langle\Theta_{i}^* \Lambda_i S_{\Lambda}^{-1}\Lambda^*_kg,f\rangle=\sum_{i\in\mathbb{N}} \langle \Lambda^*_kg,S_{\Lambda}^{-1}\Lambda_i^*\Theta_{i}f\rangle\\
&=\langle g,\Theta_k f\rangle= \langle \Theta_k^* g, f\rangle,\quad f\in \h ,g\in \K,
\end{align*}
which implies $U\Lambda^*_k=\Theta_k^*$. Also
\begin{align*}
\|f-Uf\|&=\|\sum_{i\in\mathbb{N}}  \Lambda_i^* \Lambda_i S_{\Lambda}^{-1}f-\sum_{i\in\mathbb{N}}  \Theta_i^* \Lambda_i S_{\Lambda}^{-1}f\|\\
&= \| \sum_{i\in\mathbb{N}} (\Lambda_i^*-\Theta_i^*)\Lambda_i S_{\Lambda}^{-1}f\|\\
&\leq \sum_{i\in\mathbb{N}} \|\Lambda_i-\Theta_i\| \|\Lambda_i S_{\Lambda}^{-1}\|\|f\|=\alpha\|f\|,\quad f\in M,
\end{align*}
and so we get $\|Uf\|\geq (1-\alpha)\|f\|.$
Consequently, for any finite sequnce $\{g_i\}\subseteq \K$
\begin{align*}
\|\sum\Theta_{i}^*g_i\|&=\|\sum U\Lambda_{i}^* g_i\|=\| U\sum_{i\in\mathbb{N}} \Lambda_{i}^* g_i\|\\
&\geq(1-\alpha)\|\sum\Lambda_{i}^*g_i \|\geq (1-\alpha)\sqrt{A_{\Lambda}}\big(\sum \|g_i\|^2\big)^{1/2}.
\end{align*}
\end{proof}
\begin{thm}
	Let $\Lambda=\{\Lambda_{1}T^{i-1}\in B(\h,\K); i\in \mathbb{N}\}$ be a g-Riesz sequence and for $\Theta_1\in B(\h,\K)$ there exists $\mu\in [0,1)$ such that $\|\Theta_1T^i\|\leq \mu^i \|\Theta_1\| $ and $\|\Theta_1 \| < (1-\mu) \sqrt{A_{\Lambda}}.$ Then $\{(\Lambda_{1}+\Theta_1)T^{i-1}\in B(\h,\K); i\in \mathbb{N}\}$ is a g-Riesz sequence. 
\end{thm}
\begin{proof}
	It is sufficient  to examine the conditions of Theorem \ref{j1} for the sequnce $\{(\Lambda_{1}+\Theta_1)T^{i-1}\in B(\h,\K); i\in \mathbb{N}\}$.
	\begin{align*}\label{j2}
	\sum_{i\in\mathbb{N}} \|(\Lambda_1+\Theta_1)T^{i-1} -\Lambda_{1}T^{i-1}\|^2&=\sum_{i\in\mathbb{N}}\|\Theta_1 T^{i-1}\|^2\nonumber\\
	&\leq \sum_{i\in\mathbb{N}}\mu^{2i-2}\|\Theta_1\|^2=\dfrac{\|\Theta_1\|^2}{1-\mu^2}.
	\end{align*}
	Also, by Remark \ref{remj1} 
	\begin{equation*}
	\sum_{i\in\mathbb{N}}\|\Theta_1 T^{i-1}\|\|\Lambda_{1}S_{\Lambda}^{-1}\|\leq \dfrac{\|\Theta_1\|}{(1-\mu)\sqrt{A_{\Lambda}}}<1.
	\end{equation*}
\end{proof}

\end{document}